\documentclass[reqno]{amsart}
\usepackage{amssymb,amsmath,amsthm,latexsym,booktabs,todonotes, color, comment, lineno}

\usepackage[T1]{fontenc}
\usepackage{lmodern}
\usepackage{amsmath,amssymb,amsthm,mathtools}
\usepackage{microtype}
\usepackage{geometry}
\usepackage{enumitem}
\usepackage{pdfpages}
\usepackage[hidelinks]{hyperref}

\newcommand{\Z}{\mathbb Z}
\newcommand{\F}{\mathbb F}

\newcommand{\Ker}{\operatorname{Ker}}

\newcommand{\rank}{\operatorname{rank}}

\theoremstyle{definition}
\newtheorem{definition}{Definition}

\theoremstyle{plain}
\newtheorem{lemma}[definition]{Lemma}

\newtheorem{theorem}[definition]{Theorem}

\usepackage{todonotes}

\begin{document}


\title{Corrigendum to "A CHARACTERIZATION\\ OF THE UNIT GROUP IN $\mathbb{Z}[T \times C_2]$"}
\author{Richard M. Low}
\address{Department of Mathematics and Statistics\\
         San Jose State University\\
         San Jose, CA 95192\\
         USA}
\email{richard.low@sjsu.edu}


\date{August 25, 2026. Version 0.4}

\begin{abstract}
We correct Lemma 3.6 and Theorem 3.7 of T.~Bilgin,
O.~Kusmus, and R.~M.~Low,
\emph{A Characterization of the Unit Group in
$\mathbb Z[T\times C_2]$},
Bull.\ Korean Math.\ Soc.\ \textbf{53} (2016), no. 4, 1105-1112. In particular, the corrected statement of Theorem 3.7 is that $U_1\bigl(\Z[T\times C_2]\bigr)
  \cong
  [F_{33}\rtimes F_5]\rtimes[T\times C_2].
$
\end{abstract}
\maketitle



\section{Preliminaries} \label{Prelim}

This corrigendum makes a few corrections for the paper ``A Characterization of the Unit Group in $\mathbb{Z}[T \times C_2]$.'' For the sake of completeness, we have included the original paper \cite{BKL} in the Appendix. The following framework and Lemmas 2.1--2.2, 3.1--3.5, and Theorem 2.3 found in \cite{BKL} are correct.

\vspace{10pt}

Let
\[
  T=
  \langle a,b\mid
    a^6=1,\ a^3=b^2,\ ba=a^5b
  \rangle,
\]
the nonabelian group of order $12$, and let
\[
  \rho:\Z T\longrightarrow \F_2T
\]
be the surjective ring homomorphism, where $\rho$ reduces the coefficients modulo $2$.

\vspace{10pt}

Parmenter's decomposition \cite{Parmenter} gives 
\[
  U_1(\Z T)=V\rtimes T,
\]
where
\[
  V=\langle v_1,v_2,v_3,v_4,v_5\rangle
\]
is a free group of rank $5$.  Thus
\[
  V\cong F_5.
\]

Let
\[
  M=
  \Ker\!\left(
     \rho:U(\Z T)\longrightarrow U(\F_2T)
  \right)
\]
and
\[
  M^+
  =
  M\cap U_1(\Z T)
  =
  \Ker\!\left(
     \rho\big|_{U_1(\Z T)}
  \right).
\]

Finally, set
\[
  E=
  \langle
     \rho(v_1),\rho(v_2),\rho(v_3)
  \rangle.
\]

\underline{Facts}.

\vspace{10pt}

\begin{enumerate}[label=\textnormal{(\roman*)},leftmargin=2.5em]

\item 
Lemma 3.1 gives
\begin{equation}
\label{eq:rhoV}
  \rho(V)
  =
  E\rtimes\langle a^2\rangle,
  \qquad
  E\cong C_2 \times C_2 \times C_2,
  \qquad
  |\rho(V)|=24.
\end{equation}

\item
Lemma 3.2 gives
\begin{equation}
\label{eq:rhoU1}
  \rho[U_1(\Z T)]
  =
  E\rtimes T.
\end{equation}
In particular,
\[
  E\cap T=\{1\}
\]
and
\[
  |\rho[U_1(\Z T)]|
  =
  |E|\,|T|
  =
  8\cdot12
  =
  96.
\]

\item
\[
  |\rho(V)\cap T|
  =
  \frac{|\rho(V)|\,|T|}
       {|\rho(V)T|}
  =
  \frac{24\cdot12}{96}
  =
  3.
\]
Since
\[
  \langle a^2\rangle
  \leq
  \rho(V)\cap T
\]
and $a^2$ has order $3$, it follows that
\begin{equation}
\label{eq:rhoVcapT}
\rho(V)\cap T=\langle a^2\rangle.
  \end{equation}

\item
Lemma 3.3 gives a unique canonical form for the elements of
$\rho(V)$, and Lemma 3.4 uses $E\cap T=1$ to show that
\[
  et=1,\qquad e\in E,\quad t\in T,
\]
forces $e=t=1$.\\

\item
Lemma 3.5 establishes
\begin{equation}
\label{eq:lemma35}
  M^+\leq V\rtimes\langle a^2\rangle.
\end{equation}

\end{enumerate}

\section{Corrections} \label{Correct}

We now use the facts described above to make corrections to Lemma 3.6 and Theorem 3.7 in \cite{BKL}.\\

\indent
In the published proof of Lemma 3.6 in \cite{BKL}, it is asserted that $M^+ \leq V$ and hence $M^+ = M^+ \cap V$. This is not true. Consider the following counterexample. Since $a^4\in\langle a^2\rangle\leq\rho(V)$, there exists $v\in V$ such that
\[
  \rho(v)=a^4.
\]
Then
\[
  \rho(va^2)
  =
  a^4a^2
  =
  a^6
  =
  1,
\]
so
\[
  va^2\in M^+.
\]
Its $T$-component is $a^2\neq1$, and hence, by the uniqueness of the
decomposition $V\rtimes T$,
\[
  va^2\notin V.
\]
Consequently
\[
  M^+\not\leq V.
\]

\begin{lemma}[Corrected Lemma 3.6]\label{Correct_Lem3.6}
Let
\[
  L:=M^+\cap V.
\]
Then
\[
  L\cong F_{97},
  \qquad
  [M^+:L]=3,
\]
and $M^+$ is a free group of rank $33$.  In particular,
\[
  M^+\cong F_{33}.
\]
More precisely, there is a short exact sequence
\begin{equation}
\label{eq:mainSES}
  1\longrightarrow F_{97}
  \longrightarrow M^+
  \longrightarrow C_3
  \longrightarrow1,
\end{equation}
and this extension is nonsplit.
\end{lemma}
\begin{proof}

We divide the proof into four parts.\\

\medskip
\noindent
\textbf{\underline{Step 1}.  The subgroup $L=M^+\cap V$ is free of rank $97$.}

Restrict $\rho$ to $V$:
\[
  \rho|_V:V\longrightarrow\rho(V).
\]
Its kernel is
\[
  \Ker(\rho|_V)
  =
  V\cap\Ker\!\left(
     \rho|_{U_1(\Z T)}
  \right)
  =
  V\cap M^+
  =
  L.
\]
The map is surjective onto $\rho(V)$ by definition.  Hence 
\[
  V/L\cong\rho(V).
\]
By \eqref{eq:rhoV},
\[
  |\rho(V)|=24.
\]
Therefore
\begin{equation}
\label{eq:indexVL}
  [V:L]=24.
\end{equation}

Since $V\cong F_5$, the Nielsen--Schreier Theorem yields
\[
  \rank(L)
  =
  1+[V:L](\rank(V)-1)
  =
  1+24(5-1)
  =
  97.
\]
Thus
\begin{equation}
\label{eq:Lrank97}
  L=M^+\cap V\cong F_{97}.
\end{equation}

\medskip
\noindent
\textbf{\underline{Step 2}. The subgroup $L$ has index $3$ in $M^+$.}

Let
\[
  p:V\rtimes T\longrightarrow T,
  \qquad
  p(vt)=t,
\]
be the canonical group homomorphism. Here we identify $T$ with its canonical image in $U(\mathbb{F}_2 T)$; equivalently, $\rho (t) = t$ for every $t \in T$.

We first determine $p(M^+)$.  Let $m\in M^+$ and write it uniquely as
\[
  m=vt,
  \qquad
  v\in V,\quad t\in T.
\]
Since $\rho(m)=1$,
\[
  \rho(v)t=1,
\]
so
\[
  t^{-1}=\rho(v)\in\rho(V)\cap T.
\]
Using \eqref{eq:rhoVcapT},
\[
  t\in\langle a^2\rangle.
\]
Hence
\[
  p(M^+)\leq\langle a^2\rangle.
\]

Conversely, take any
\[
  t\in\langle a^2\rangle.
\]
By \eqref{eq:rhoV},
\[
  t^{-1}\in\langle a^2\rangle\leq\rho(V).
\]
Therefore there exists $v\in V$ such that
\[
  \rho(v)=t^{-1}.
\]
It follows that
\[
  \rho(vt)
  =
  \rho(v)t
  =
  t^{-1}t
  =
  1.
\]
Thus
\[
  vt\in M^+
\]
and
\[
  p(vt)=t.
\]
Consequently
\begin{equation}
\label{eq:imageProjection}
  p(M^+)=\langle a^2\rangle\cong C_3.
\end{equation}

The kernel of the restriction
\[
  p|_{M^+}:M^+\longrightarrow\langle a^2\rangle
\]
is
\[
  \Ker(p|_{M^+})
  =
  M^+\cap V
  =
  L.
\]
Hence
\begin{equation}
\label{eq:SESML}
  1\longrightarrow L
  \longrightarrow M^+
  \overset{p|_{M^+}}{\longrightarrow}
  \langle a^2\rangle
  \longrightarrow1
\end{equation}
is exact.  Therefore
\begin{equation}
\label{eq:indexML}
  [M^+:L]=3.
\end{equation}

Combining \eqref{eq:Lrank97} and \eqref{eq:SESML} gives
\[
  1\longrightarrow F_{97}
  \longrightarrow M^+
  \longrightarrow C_3
  \longrightarrow1.
\]

\medskip
\noindent
\textbf{\underline{Step 3}.  The group $M^+$ is torsion-free.}

Since $L\cong F_{97}$, the subgroup $L$ is torsion-free. Suppose that $x\in M^+$ has finite order $n$.  If the image of $x$
under the quotient map
\[
  M^+\longrightarrow M^+/L\cong C_3
\]
is trivial, then $x\in L$, and therefore $x=1$.

Assume instead that the image of $x$ is nontrivial.  Its order in
$C_3$ is then $3$.  The order of the image divides $n$, so $3\mid n$.
Set
\[
  y=x^{n/3}.
\]
The order of $y$ is exactly $3$, and
\[
  y\in M^+.
\]

We now prove that no nonidentity element of order $3$ can belong to
$M^+$. Since $y\in M^+$,
\[
  \rho(y)=1.
\]
Equivalently,
\[
  y-1\in2\Z T.
\]
Thus there exists $P_1\in\Z T$ such that
\[
  y=1+2P_1.
\]

We claim inductively that for every integer $k\geq1$ there exists
$P_k\in\Z T$ satisfying
\begin{equation}
\label{eq:2adicInduction}
  y=1+2^kP_k.
\end{equation}

The case $k=1$ has already been established.  Suppose
\eqref{eq:2adicInduction} holds for some $k\geq1$.  Since $y^3=1$,
\[
  (1+2^kP_k)^3=1.
\]
Although $\Z T$ need not be commutative, the usual binomial expansion
is valid here because the identity element commutes with $P_k$.
Therefore
\[
  1
  +3\cdot2^kP_k
  +3\cdot2^{2k}P_k^2
  +2^{3k}P_k^3
  =
  1.
\]
So in the torsion-free
additive group of $\Z T$, we obtain
\[
  3P_k
  +3\cdot2^kP_k^2
  +2^{2k}P_k^3
  =
  0.
\]
Reduce the coefficients modulo $2$.  Since $k\geq1$, the last two
terms vanish modulo $2$, while $3\equiv1\pmod2$.  Hence
\[
  \rho(P_k)=0.
\]
Thus
\[
  P_k\in2\Z T.
\]
Write
\[
  P_k=2P_{k+1}.
\]
Then
\[
  y
  =
  1+2^kP_k
  =
  1+2^{k+1}P_{k+1},
\]
which proves the induction.

Consequently
\[
  y-1\in2^k\Z T
  \qquad
  \text{for every }k\geq1.
\]
Since $|T|=12$, the additive group of $\Z T$ is free abelian of
rank $12$:
\[
  (\Z T,+)\cong\Z^{12}.
\]
Therefore
\[
  \bigcap_{k\geq1}2^k\Z T=\{0\}.
\]
It follows that
\[
  y-1=0,
\]
so $y=1$, contradicting the fact that $y$ has order $3$.

Thus no nontrivial finite-order element can occur in $M^+$:
\begin{equation}
\label{eq:MplusTorsionFree}
  M^+\text{ is torsion-free}.
\end{equation}

\medskip
\noindent
\textbf{\underline{Step 4}.  The group $M^+$ is free of rank $33$.}

By \eqref{eq:Lrank97} and \eqref{eq:indexML}, $M^+$ contains the free
group
\[
  L\cong F_{97}
\]
as a subgroup of finite index $3$.  Thus $M^+$ is virtually free.

A standard consequence of Bass--Serre theory \cite{Serre} is that every torsion-free virtually free group is free. Hence, by \eqref{eq:MplusTorsionFree},
\[
  M^+\cong F_r
\]
for some $r$.

Apply the Nielsen--Schreier Theorem to the index-$3$ subgroup
\[
  L\leq M^+.
\]
Using
\[
  \rank(L)=97
  \qquad\text{and}\qquad
  [M^+:L]=3,
\]
we obtain
\[
  97
  =
  1+3(r-1).
\]
Thus, 
\[
  r=33
\]
and 
\[
  M^+\cong F_{33}.
\]

Finally, the extension
\[
  1\longrightarrow F_{97}
  \longrightarrow M^+
  \longrightarrow C_3
  \longrightarrow1
\]
cannot split.  If it split, $M^+$ would contain a subgroup isomorphic
to $C_3$, contradicting the torsion-freeness of $M^+$.
\end{proof}

\begin{theorem}[Corrected Theorem 3.7]\label{Correct_Theorem3.7}

Let $T^* = T \times C_2$, where $T = \langle a, b: a^6 = 1, a^3 = b^2, 
ba = a^5b \rangle$. Then, $U_1(\mathbb{Z}T^*) \cong [F_{33} \rtimes 
F_5] \rtimes T^*$, where $F_i$ is a free group of rank $i$.
\end{theorem}
\begin{proof}
Invoking Theorem 2.3 in \cite{BKL}, we obtain $U_1(\mathbb{Z}[T \times C_2]) = K \rtimes (V \rtimes T) \cong M \rtimes (V \rtimes T) = [M^+ \times C_2] \rtimes (V \rtimes T) = [M^+ \rtimes V] \rtimes (T \times C_2) = [F_{33} \rtimes F_5] \rtimes (T \times C_2)$, where $F_i$ is a free group of rank $i$.
\end{proof}

\section{Tool and computational resource disclosure}\label{Disclosure}
The author supports the Leiden Declaration on Artificial Intelligence and Mathematics. ChatGPT was used to proofread the manuscript for grammar, spelling, and punctuation.

\section{Appendix}\label{Appendix}

For the convenience of the reader, we have included the entire original paper \cite{BKL} in this appendix. 

\vspace{10pt}



\section*{Supplementary material}

# A CHARACTERIZATION OF THE UNIT GROUP IN $\mathbb{Z}[T \times C_2]$

Tevfik Bilgin, Omer Kusmus, and Richard M. Low

Abstract. Describing the group of units $U(\mathbb{Z}G)$ of the integral group ring $\mathbb{Z}G$, for a finite group $G$, is a classical and open problem. In this note, we show that $U_1(\mathbb{Z}[T \times C_2]) \cong [F_{97} \rtimes F_5] \rtimes [T \times C_2]$, where $T = \langle a, b : a^6 = 1, a^3 = b^2, ba = a^5 b\rangle$ and $F_{97}$, $F_5$ are free groups of ranks 97 and 5, respectively.

## 1. Introduction

Given a finite group $G$ and the ring of integers $\mathbb{Z}$, we denote the integral group ring as $\mathbb{Z}G$. Its elements are all finite formal sums

$$\sum_{g \in G} r_g g, \text{ where } r_g \in \mathbb{Z}.$$

There is a surjective ring homomorphism $\epsilon$: $\mathbb{Z}G \to \mathbb{Z}$, defined by

$$\sum_{g \in G} r_g g \longmapsto \sum_{g \in G} r_g.$$

The ring homomorphism $\epsilon$ is called the *augmentation map* and its kernel $\Delta_{\mathbb{Z}}(G) = \langle g - 1 : g \in G\rangle$ is the *augmentation ideal.* We will denote the group of units of $\mathbb{Z}G$ by $U(\mathbb{Z}G)$. $U_1(\mathbb{Z}G)$ will denote the units of augmentation one in $U(\mathbb{Z}G)$. Thus, $U_1(\mathbb{Z}G)$ is a normal subgroup of $U(\mathbb{Z}G)$ and $\pm U_1(\mathbb{Z}G) = U(\mathbb{Z}G)$. Observe that $\pm G \leq U(\mathbb{Z}G)$. The elements $\pm G$ are called the *trivial units* of $\mathbb{Z}G$.

Describing the units of the integral group ring is a classical and difficult problem. Over the years, it has drawn the attention of those working in the areas of algebra, number theory, and algebraic topology. Most descriptions of $U(\mathbb{Z}G)$ in the mathematical literature either give an explicit description of the units, the general structure of $U(\mathbb{Z}G)$, or a subgroup of finite index of the unit group $U(\mathbb{Z}G)$. These results were often obtained by using techniques from representation theory and algebraic number theory.

---

In 1940, substantial work on the unit problem was done by Graham Higman [5, 6]. He first showed that if $U(\mathbb{Z}G) = \pm G$, then $U(\mathbb{Z}[G \times C_2]) = \pm(G \times C_2)$. Using this, he showed that $U(\mathbb{Z}G) = \pm G \iff$ G is abelian of exponent 2, 3, 4, or 6 or $G = E \times K_8$ where $K_8$ is the quaternion group of order 8 and $E$ is an elementary abelian 2–group. Furthermore, Higman gave a general structure theorem for $U(\mathbb{Z}A)$, where $A$ is a finite abelian group. Other results include: $A_4$ and $S_4$ by Allen-Hobby [1, 2], $D_{2p}$ by Passman-Smith [21], $G = C_p \rtimes C_q$, where $q$ is a prime dividing $p-1$ by Galovitch-Reiner-Ullom [4], $|G| = p^3$ by Ritter-Sehgal [23], and $U(\mathbb{Z}S_3)$ by Hughes-Pearson [7]. Jespers and Parmenter [10] gave a more explicit description of $U(\mathbb{Z}S_3)$. In 1993, Jespers and Parmenter [11] completed the description of $U(\mathbb{Z}G)$ for all groups of order 16. Jespers [9], in 1995, gave a description of $U(\mathbb{Z}G)$, for the dihedral group of order 12 and for $G = D_8 \times C_2$. More recently, Bilgin [3] gave a characterization of $U_1(\mathbb{Z}C_{12})$. Kusmus and Denizler [15] gave a construction of $U(\mathbb{Z}C_{24})$. Kelebek and Bilgin [14] described the structure of $U_1(\mathbb{Z}[C_n \times K_4])$. The interested reader is directed to Sehgal's [24] comprehensive survey on the unit problem in integral group rings.

In [17, 18], a general algebraic framework was developed to study $U(\mathbb{Z}G^*)$, where $G^* = G \times C_p$ with $p$ prime. In the following sections of this note, we focus on the case where $p = 2$ and then resolve a conjecture found in [17].

## 2. $U(\mathbb{Z}[G \times C_2])$

Here, we obtain a result which helps us to answer the following question: Assuming that we have a good description of $U(\mathbb{Z}G)$, can we obtain a description of $U(\mathbb{Z}G^*)$, where $G^* = G \times C_2$?

Let $G^* = G \times \langle x \rangle$, $x^2 = 1$, with $|G| = n$. Decomposing $G^*$ into two cosets, we have that $G^* = G \cup xG = \{g_1, g_2, \ldots, g_n, xg_1, \ldots, xg_n\}$. Thus, $\mathbb{Z}G^* = \mathbb{Z}G \oplus x\mathbb{Z}G$, a direct sum of abelian groups. Here, the equal sign denotes equality as sets. Now, consider the surjective group homomorphism $\pi : G^* \to G$ defined by $g \mapsto g, x \mapsto 1$. This induces a ring homomorphism $\pi : \mathbb{Z}G^* \to \mathbb{Z}G$; where $\pi(P_1 + xP_2) = P_1 + P_2$, and $P_1$, $P_2 \in \mathbb{Z}G$. At the ring level, $\mathrm{Ker}(\pi) = K^* = (x-1)\mathbb{Z}G$. So, we have the sequence of maps

$$K^* \xrightarrow{\ \iota\ } \mathbb{Z}G^* \xrightarrow{\ \pi\ } \mathbb{Z}G.$$

Restricting $\pi$ to the group of units, we obtain the split exact sequence of groups:

$$K \xrightarrow{\ \iota\ } U(\mathbb{Z}G^*) \xrightarrow{\ \pi\ } U(\mathbb{Z}G),$$

where $K = \mathrm{Ker}(\pi)$. Hence, $U(\mathbb{Z}G^*) = K \rtimes U(\mathbb{Z}G)$. Note that $K = U(\mathbb{Z}G^*) \cap (1 + K^*)$. Thus, a unit in $K$ has the form $1 + (x-1)P$, where $P \in \mathbb{Z}G$, and has an inverse $1 + (x-1)Q$, where $Q \in \mathbb{Z}G$.

Also, let us consider the surjective ring homomorphism $\rho : \mathbb{Z}G \twoheadrightarrow \mathbb{Z}_2G$, where $\rho$ reduces the coefficients modulo 2. The kernel of $\rho$, say $M^*$ (as an ideal), is $M^* = 2\mathbb{Z}G$. Thus, we have the following sequence of maps:

$$M^* \xrightarrow{\ \iota\ } \mathbb{Z}G \xrightarrow{\ \rho\ } \mathbb{Z}_2G.$$

Furthermore, $\rho$ induces the following exact sequence of groups, which does not necessarily split:

$$M \xrightarrow{\ \iota\ } U(\mathbb{Z}G) \xrightarrow{\ \rho\ } U(\mathbb{Z}_2G),$$

where $M$ is the kernel of the group homomorphism $\rho$. Observe that $M = U(\mathbb{Z}G) \cap (1 + M^*)$. Thus, a unit in $M$ has the form $1 + 2P$, where $P \in \mathbb{Z}G$ and has an inverse $1 + 2Q$, where $Q \in \mathbb{Z}G$. Notice that here at the group level, $\rho$ is not necessarily surjective.

Since $G^* = G \times \langle x \rangle$ and $x^2 = 1$, we have the group homomorphism $\sigma : G^* \to U(\mathbb{Z}G)$, where $\sigma(g) = g$ and $\sigma(x) = -1$. This extends to a ring homomorphism $\sigma : \mathbb{Z}G^* \to \mathbb{Z}G$. So, we have the following diagram of rings:

$$\begin{array}{ccccc} K^* & \xrightarrow{\ \iota\ } & \mathbb{Z}G^* & \xrightarrow{\ \pi\ } & \mathbb{Z}G \\ \sigma\big\downarrow & & \sigma\big\downarrow & & \rho\big\downarrow \\ M^* & \xrightarrow{\ \iota\ } & \mathbb{Z}G & \xrightarrow{\ \rho\ } & \mathbb{Z}_2G. \end{array}$$

Observe that $\rho \circ \pi = \rho \circ \sigma$. Hence, $\sigma(K^*) \subseteq M^*$. Note that $\sigma$ maps the element $1 + (x-1)P \in K$ to the element $1 - 2P \in M$, where $P \in \mathbb{Z}G$. Thus, $\sigma(K) \subseteq M$.

**Lemma 2.1.** *Let $G^* = G \times \langle x \rangle$, where $x$ has order 2, $u = 1 + (x-1)P$, $v = 1 + (x-1)Q$, where $P$, $Q \in \mathbb{Z}G$. Then $u$ and $v$ are multiplicative inverses of each other in $K \iff 1 - 2P$ and $1 - 2Q$ are multiplicative inverses of each other in $U(\mathbb{Z}G)$.*

*Proof.* Let $u$, $v \in K$; with $uv = 1$. It is straightforward to see that $uv = 1 + (x-1)(P + Q - 2PQ)$.

$$\begin{aligned} \text{Hence, } uv = 1 &\iff (x-1)(P+Q-2PQ) = 0 \\ &\iff (2PQ - P - Q) + (P + Q - 2PQ)x = 0 \\ &\iff 2PQ - P - Q = 0 \\ &\iff 4PQ - 2P - 2Q = 0 \\ &\iff 1 - 2P - 2Q + 4PQ = 1 \\ &\iff (1-2P)(1-2Q) = 1. \end{aligned}$$

□

**Lemma 2.2.** *The map $\sigma : K \to M$ is an isomorphism of groups.*

*Proof.* Note that $\sigma$ maps the element $1 + (x-1)P$ of $K$ to the element $1 - 2P$ of $M$. It is then easy to show that $\sigma$ is injective. It follows from Lemma 2.1 that $\sigma$ is surjective. □

Summarizing, we have the following diagram of groups:

$$\begin{array}{ccccc} K & \xrightarrow{\ \iota\ } & U(\mathbb{Z}G^*) & \xrightarrow{\ \pi\ } & U(\mathbb{Z}G) \\ \cong\big\downarrow & & \sigma\big\downarrow & & \rho\big\downarrow \\ M & \xrightarrow{\ \iota\ } & U(\mathbb{Z}G) & \xrightarrow{\ \rho\ } & U(\mathbb{Z}_2G). \end{array}$$

**Theorem 2.3.** *$U(\mathbb{Z}G^*) = K \rtimes U(\mathbb{Z}G) \cong M \rtimes U(\mathbb{Z}G)$.*

*Proof.* The elements of the semi-direct product $M \rtimes U(\mathbb{Z}G)$ should be viewed as ordered pairs $(u, w)$, where $u \in M$ and $w \in U(\mathbb{Z}G)$. If $k \in K$ and $w \in U(\mathbb{Z}G)$, then the isomorphism maps $kw$ to $(\sigma(k), w)$ with the action of $U(\mathbb{Z}G)$ on $M$ induced by conjugation in $U(\mathbb{Z}G)$. □

The problem of describing $U(\mathbb{Z}G^*)$ has been reduced to the problem of describing $M$. In the next section, we apply Theorem 2.3 and resolve a conjecture involving $U(\mathbb{Z}G^*)$, where $G^*$ is a particular non-abelian group of order 24.

## 3. Resolution of a conjecture

It was shown by Jespers [8] that there are only four finite groups $G$ with the property that $G$ has a non-abelian free normal complement in $U_1(\mathbb{Z}G)$, namely $G = S_3$, $D_4$ (the dihedral group of order 8), $P = \langle a, b : a^4 = 1 = b^4, bab^{-1}a^{-1} = a^2\rangle$, and the non-abelian group $T$ (of order 12) described by the presentation

$$T = \langle a, b : a^6 = 1, a^3 = b^2, ba = a^5 b\rangle.$$

In [9, 11, 17], the structure of $U(\mathbb{Z}[G \times C_2])$ is determined for $G = S_3, D_4$ and $P$. In this section, we disprove the following conjecture, first posed in [17]:

**Conjecture.** *Let $T^* = T \times C_2$, where $T = \langle a, b : a^6 = 1, a^3 = b^2, ba = a^5 b\rangle$. Then, $U_1(\mathbb{Z}T^*) \cong [F_{33} \rtimes F_5] \rtimes T^*$, where $F_i$ is a free group of rank $i$.*

This is certainly a plausible conjecture. Later, it was shown in [12] that $U(\mathbb{Z}[T \times C_2])$ is commensurable with a free-by-free group. We will show that if $F_{33}$ is replaced with $F_{97}$, then a correct result is obtained.

In 1993, Parmenter [20] showed that $U_1(\mathbb{Z}T) = V \rtimes T$, where $V = \langle v_1, v_2, v_3, v_4, v_5\rangle$ is a free group of rank five. He also gave the generators of $V$ to be:

$$\begin{aligned}
v_1 &= 1 + (1 + a^3)(-a^2 + ba^2)(1 - a^2),\\
v_2 &= 1 + (1 + a^3)(-a^2 + ba)(1 - a^2),\\
v_3 &= 1 + (1 + a^3)(-a^2 + b)(1 - a^2),\\
v_4 &= 1 + [-1 + (1 + a^3)a^2(a^2 + ba^2)](1 - a^2),\\
v_5 &= 1 + [-1 - a^2 + (1 + a^3)a(1 - a - 2ba^2)](1 - a^2).
\end{aligned}$$

Let us determine $\rho(V)$. It is straight-forward to verify the following facts. First, $\rho(v_i)\rho(v_j) = \rho(v_j)\rho(v_i)$, where $1 \le i, j \le 3$. Also, $\rho(v_1)^2 = \rho(v_2)^2 = \rho(v_3)^2 = 1$ and thus, $E = \langle \rho(v_1), \rho(v_2), \rho(v_3)\rangle \cong C_2 \times C_2 \times C_2$. Now, calculations show that $a^2\rho(v_1)a^4 = \rho(v_2), a^2\rho(v_2)a^4 = \rho(v_3), a^2\rho(v_3)a^4 = \rho(v_1), a^2\rho(v_1) = \rho(v_4)$, and $[\rho(v_4)]^3 a^4 = \rho(v_5)$. Thus, $\langle a^2, \rho(v_1)\rangle = \langle a^2, \rho(v_2)\rangle = \langle a^2, \rho(v_3)\rangle = \rho(V)$.

**Lemma 3.1.** *$\rho(V) = E \rtimes \langle a^2\rangle$, a group of order 24.*

*Proof.* Since $E = \langle \rho(v_1), \rho(v_2), \rho(v_3)\rangle$ is normalized by $\rho(v_1), \rho(v_2), \rho(v_3)$, and $a^2$, we have that $E \trianglelefteq \rho(V)$. So, $E \cdot \langle a^2\rangle \le \rho(V)$. In fact, $E \cdot \langle a^2\rangle = \rho(V)$ and

$E \cap \langle a^2 \rangle = 1$. Thus, $\rho(V) = E \rtimes \langle a^2 \rangle \cong [C_2 \times C_2 \times C_2] \rtimes C_3$, a group of order 24. □

**Lemma 3.2.** *$\rho[U_1(\mathbb{Z}T)] = E \rtimes T$.*

*Proof.* Clearly, $\rho[U_1(\mathbb{Z}T)] = \rho(V \rtimes T) = \rho(V) \cdot T$. Since $E \trianglelefteq \rho(V)$, we have that $\rho[U_1(\mathbb{Z}T)] = E \cdot T$. Since $a\rho(v_1)a^5 = \rho(v_3), a\rho(v_3)a^5 = \rho(v_2), a\rho(v_2)a^5 = \rho(v_1), b\rho(v_1)b^3 = \rho(v_1), b\rho(v_2)b^3 = \rho(v_3), b\rho(v_3)b^3 = \rho(v_2)$, we see that $E$ is normalized by $T$. Note that $E \cap T = 1$. Hence, the lemma is established. □

A remark should be made at this point. Since $\rho[U_1(\mathbb{Z}T)]$ has order 96, $|\rho(V)| = 24$, and $|T| = 12$, this implies that $|\rho(V) \cap T| = 3$. But $\langle a^2 \rangle \leq \rho(V) \cap T$, where the order of $a^2$ is 3. Hence, $\rho(V) \cap T = \langle a^2 \rangle$. Now, we have the diagram:

$$
\begin{array}{ccccc}
K & \xrightarrow{\iota} & U(\mathbb{Z}T^*) & \xrightarrow{\pi} & U(\mathbb{Z}T) \\
\cong\downarrow & & \sigma\downarrow & & \rho\downarrow \\
M & \xrightarrow{\iota} & U(\mathbb{Z}T) & \xrightarrow{\rho} & U(\mathbb{Z}_2T) \\
\iota\uparrow & & \iota\uparrow & & \iota\uparrow \\
M^+ & \xrightarrow{\iota} & U_1(\mathbb{Z}T) & \xrightarrow{\text{onto}} & E \rtimes T \\
\iota\uparrow & & \iota\uparrow & & \iota\uparrow \\
M^+ \cap V & \xrightarrow{\iota} & V & \xrightarrow{\text{onto}} & E \rtimes \langle a^2 \rangle.
\end{array}
$$

**Lemma 3.3.** *$\rho(V) = \langle \rho(v_1), \rho(v_2), \rho(v_3), \rho(v_4) \rangle = \{[\rho(v_1)]^{i_1} \cdot [\rho(v_2)]^{i_2} \cdot [\rho(v_3)]^{i_3} \cdot [\rho(v_4)]^{i_4} : 0 \leq i_1, i_2, i_3 \leq 1; 0 \leq i_4 \leq 2\}$. Furthermore, this canonical representation is unique.*

*Proof.* Note that $\rho(V) = \langle a^2, \rho(v_1) \rangle = \langle \rho(v_1), \rho(v_2), \rho(v_3), \rho(v_4) \rangle$. Also, calculations show the following:

$\rho(v_4)\rho(v_1) = \rho(v_2)a^2\rho(v_1) = \rho(v_2)^2a^2 = a^2 = \rho(v_2)\rho(v_4)$,

$\rho(v_4)\rho(v_2) = \rho(v_2)a^2\rho(v_2) = \rho(v_2)\rho(v_3)a^2 = \rho(v_2)\rho(v_3)\rho(v_2)\rho(v_4) = \rho(v_3)\rho(v_4)$,

$\rho(v_4)\rho(v_3) = \rho(v_2)a^2\rho(v_3) = \rho(v_2)\rho(v_1)a^2 = \rho(v_2)\rho(v_1)\rho(v_2)\rho(v_4) = \rho(v_1)\rho(v_4)$,

$\rho(v_4)^2\rho(v_1) = \rho(v_3)\rho(v_4)^2$,

$\rho(v_4)^2\rho(v_2) = \rho(v_1)\rho(v_4)^2$,

$\rho(v_4)^2\rho(v_3) = \rho(v_2)\rho(v_4)^2$,

$\rho(v_4)^3\rho(v_1) = \rho(v_2)\rho(v_3)$,

$\rho(v_4)^3\rho(v_2) = \rho(v_1)\rho(v_3)$,

$\rho(v_4)^3\rho(v_3) = \rho(v_1)\rho(v_2)$,

$\rho(v_4)^4\rho(v_1) = \rho(v_4)\rho(v_4)^3\rho(v_1) = \rho(v_1)\rho(v_3)\rho(v_4)$,

$\rho(v_4)^4\rho(v_2) = \rho(v_4)\rho(v_4)^3\rho(v_2) = \rho(v_1)\rho(v_2)\rho(v_4)$,

$\rho(v_4)^4\rho(v_3) = \rho(v_4)\rho(v_4)^3\rho(v_3) = \rho(v_2)\rho(v_3)\rho(v_4)$,

$\rho(v_4)^5\rho(v_1) = \rho(v_1)\rho(v_2)\rho(v_3)\rho(v_4)^2\rho(v_1) = \rho(v_1)\rho(v_2)\rho(v_4)^2$,

$\rho(v_4)^5\rho(v_2) = \rho(v_1)\rho(v_2)\rho(v_3)\rho(v_4)^2\rho(v_2) = \rho(v_2)\rho(v_3)\rho(v_4)^2$,

$\rho(v_4)^5\rho(v_3) = \rho(v_1)\rho(v_2)\rho(v_3)\rho(v_4)^2\rho(v_3) = \rho(v_1)\rho(v_3)\rho(v_4)^2$.

Thus, every word in $\rho(V)$ can be put into the canonical form $[\rho(v_1)]^{i_1} \cdot [\rho(v_2)]^{i_2} \cdot [\rho(v_3)]^{i_3} \cdot [\rho(v_4)]^{i_4}$, where $0 \le i_1, i_2, i_3 \le 1$ and $0 \le i_4 \le 2$. This representation is unique, since $|\rho(V)| = 24$. □

**Lemma 3.4.** *Let $w[\rho(v_1), \rho(v_2), \rho(v_3)] \in E, t \in T$, with $w[\rho(v_1), \rho(v_2), \rho(v_3)] \cdot t = 1$. Then, $t = 1_T$.*

*Proof.* Suppose that $w[\rho(v_1), \rho(v_2), \rho(v_3)] \cdot t = 1 = E \cap T$. Then, we have $w[\rho(v_1),\ \rho(v_2), \rho(v_3)] = t^{-1} \in T$ and $w[\rho(v_1), \rho(v_2), \rho(v_3)] \in E$. This implies that $w[\rho(v_1),\ \rho(v_2), \rho(v_3)] \in E \cap T = 1_T$. Thus, $w[\rho(v_1), \rho(v_2), \rho(v_3)] = 1$, which implies that $t^{-1} = 1_T$. Hence, $t = 1_T$. □

**Lemma 3.5.** $M^+ \le V \rtimes \langle a^2 \rangle$.

*Proof.* Suppose that $w(v_1, v_2, v_3, v_4, v_5) \cdot t \in M^+$, where $t \in T$. This implies that $\rho[w(v_1, v_2, v_3, v_4, v_5) \cdot t] = 1$. By Lemma 3.3, we have that $([\rho(v_1)]^{i_1} \cdot [\rho(v_2)]^{i_2} \cdot [\rho(v_3)]^{i_3} \cdot [\rho(v_4)]^{i_4}) \cdot t = 1$, where $0 \le i_1, i_2, i_3 \le 1; 0 \le i_4 \le 2; t \in T$. Now, $[\rho(v_1)]^{i_1} \cdot [\rho(v_2)]^{i_2} \cdot [\rho(v_3)]^{i_3} \cdot [\rho(v_4)]^{i_4}$ has three possible forms:

$$\begin{cases} \rho(v_1)^{i_1} \cdot \rho(v_2)^{i_2} \cdot \rho(v_3)^{i_3} \cdot [\rho(v_2)a^2], & \text{if } i_4 = 1; \\ \rho(v_1)^{i_1} \cdot \rho(v_2)^{i_2} \cdot \rho(v_3)^{i_3} \cdot [\rho(v_2)\rho(v_3)a^4], & \text{if } i_4 = 2; \\ \rho(v_1)^{i_1} \cdot \rho(v_2)^{i_2} \cdot \rho(v_3)^{i_3}, & \text{if } i_4 = 0. \end{cases}$$

Using Lemma 3.4, we have $([\rho(v_1)]^{i_1} \cdot [\rho(v_2)]^{i_2} \cdot [\rho(v_3)]^{i_3} \cdot [\rho(v_4)]^{i_4}) \cdot t = 1$ implies that $t \in \langle a^2 \rangle$. □

**Lemma 3.6.** *$M^+$ is a free group of rank 97.*

*Proof.* Since $M^+ \le \rho^{-1}[E \rtimes T]$, where $E = \langle \rho(v_1), \rho(v_2), \rho(v_3) \rangle \cong C_2 \times C_2 \times C_2$, we see that $M^+$ consists of the elements of the form

$$\rho^{-1}[\rho(v_1)^{j_1}\rho(v_2)^{j_2}\rho(v_3)^{j_3} \cdot t],$$

where $0 \le j_1, j_2, j_3 \le 1$ and $t \in T$. Since $M^+$ is an appropriate kernel of $\rho$, then $\rho(M^+) = 1$. If we consider an element in $M^+$ as

$$\alpha = \rho^{-1}[\rho(v_1)^{j_1}\rho(v_2)^{j_2}\rho(v_3)^{j_3} \cdot t],$$

we see that $\rho(\alpha) = \rho(v_1)^{j_1}\rho(v_2)^{j_2}\rho(v_3)^{j_3} \cdot t = 1$. By Lemma 3.4, $t = 1$. This implies that $M^+$ consists of elements of the form

$$\rho^{-1}[\rho(v_1)^{j_1}\rho(v_2)^{j_2}\rho(v_3)^{j_3}] \in V.$$

Thus, $M^+ \le V$. Since $V$ is a free group, the Nielson-Schreier Theorem states that $M^+$ is a free group. Note that $M^+ = M^+ \cap V$. Now, consider the induced isomorphism $\bar{\rho} : \frac{V}{M^+ \cap V} \twoheadrightarrow E \rtimes \langle a^2 \rangle$, which implies that $[V : M^+ \cap V] = [V :$

$M^+] = 24$. Since $V$ is a free group of rank 5, this implies that $M^+$ is a free group of rank $(24)(5) - 24 + 1 = 97$. □

**Theorem 3.7.** *Let $T^* = T \times C_2$, where $T = \langle a, b : a^6 = 1, a^3 = b^2, ba = a^5 b\rangle$. Then, $U_1(\mathbb{Z}T^*) \cong [F_{97} \rtimes F_5] \rtimes T^*$, where $F_i$ is a free group of rank $i$.*

*Proof.* Invoking Theorem 2.3, we obtain $U_1(\mathbb{Z}[T \times C_2]) = K \rtimes (V \rtimes T) \cong M \rtimes (V \rtimes T) = [M^+ \times C_2] \rtimes (V \rtimes T) = [M^+ \rtimes V] \rtimes (T \times C_2) = [F_{97} \rtimes F_5] \rtimes (T \times C_2)$, where $F_i$ is a free group of rank $i$. □

**Acknowledgments.** The authors are grateful for the valuable comments made by the referee.

Tevfik Bilgin
Department of Mathematics
Fatih University
Istanbul, Turkey
*E-mail address*: tbilgin@fatih.edu.tr

Omer Kusmus
Department of Mathematics
Yüzüncü Yıl University
Van, Turkey
*E-mail address*: omerkusmus@yyu.edu.tr

Richard M. Low
Department of Mathematics
San Jose State University
San Jose, CA 95192, USA
*E-mail address*: richard.low@sjsu.edu



\end{document}